\documentclass[reqno,12pt]{amsproc}
\usepackage{ucs}
\usepackage[english]{babel}
\usepackage{amssymb,amsmath}
\usepackage{indentfirst}
\usepackage[usenames]{color}
\usepackage{graphicx}
\usepackage[left=2cm,right=2cm,top=1.5cm,bottom=1.5cm,bindingoffset=0cm]{geometry}
\usepackage{amsthm}
\usepackage{amsfonts}
\usepackage{graphicx}
\graphicspath{{pictures/}}
\DeclareGraphicsExtensions{.pdf,.png,.jpg}
\usepackage{wrapfig}

\theoremstyle{plain}  
\newtheorem*{acknowledgements}{Acknowledgements}
\newtheorem{theorem}{\bf Theorem}
\newtheorem{lemma}{\bf Lemma}
\newtheorem{corollary}{\bf Corollary}

\renewcommand{\theoldthm}{\Alph{oldthm}}

\renewcommand{\theoldlemma}{\Alph{oldlemma}}
\newtheorem{remark}{\bf Remark}

\renewcommand{\Im}{\mathop{\mathrm{Im}}\nolimits}
\renewcommand{\Re}{\mathop{\mathrm{Re}}\nolimits}
\newcommand{\g}{\gamma}
\renewcommand{\a}{\alpha}
\renewcommand{\b}{\beta}
\renewcommand{\t}{\tau}
\newcommand{\res}{\operatorname{{\sf res}}}

\usepackage{xcolor}
\usepackage{hyperref}

\hypersetup{pdfstartview=FitH, linkcolor=linkcolor,urlcolor=urlcolor, colorlinks=true, linkcolor=blue, citecolor=blue}

\begin{document}

\title{Cosine polynomials with restrictions \\on their algebraic representation}
\author{Kristina Oganesyan}
\address{Centre de Recerca Matem\`atica and Universitat Aut\`onoma de Barcelona, Lomonosov Moscow State University and Moscow Center for fundamental and applied mathematics}
\email{oganchris@gmail.com}
\thanks{The work was supported by the Moebius Contest Foundation for Young Scientists and the Foundation for the Advancement of Theoretical Physics and Mathematics $``$BASIS$"$ (grant no 19-8-2-28-1).}
\date{}

\begin{abstract}
We prove that for any even algebraic polynomial $p$ one can find a cosine polynomial with an arbitrary small $l_1$-norm of coefficients such that the first coefficients of its representation as an algebraic polynomial in $\cos x$ coincide with those of $p$.
\end{abstract}
\keywords{Chebyshev polynomials, Vandermonde matrices, $l_1$-norm.}
\maketitle

\section{Introduction} 
The work is devoted to the following problem. Suppose we are given a cosine polynomial $\sum_{k=0}^n a_k\cos kx,\;a_k\in\mathbb{R}$. With the help of Chebyshev polynomials $T_k(x)$, we can rewrite it as $\sum_{k=0}^n b_k\cos ^kx$, an algebraic polynomial in $\cos x$, according to the equality $\cos kx=T_k(\cos x)$. Conversely, any algebraic polynomial in $\cos x$ can be represented as a trigonometric one. So one can pass from one of these representations to another choosing the most appropriate one of the bases $\{\cos kx\}_{k=0}^{\infty}$ and $\{\cos^kx\}_{k=0}^{\infty}$. But what if we look at all the cosine polynomials whose certain coordinates with respect to the basis $\{\cos^kx\}_{k=0}$ are fixed? In other words, if we fix some $K\subset \mathbb{N}$, some numbers $\{c_k\}_{k\in K}$ and consider
\begin{align*}
A(K,\{c_k\}):=\Big\{\{a_k\}_{k=0}^n,\;n\in\mathbb{N}:\;\sum_{k=0}^n a_k\cos kx\equiv \sum_{k=0}^n b_k\cos ^kx,\;b_k=c_k\;\text{for}\;k\in K\Big\},
\end{align*}  
what can we say about $\sum_{k=0}^n|a_k|$ if we know that $\{a_k\}_{k=0}^{n}\in A(K,\{c_k\})$? Can we find a trigonometric polynomial belonging to $A(K,\{c_k\})$ with $``$small$"$ $l_1$-norm of the coefficients? If the answer is positive, then it is possible to adjust any $\sum_{k\in K}c_k\cos^kx$ by adding a small trigonometric polynomial so that the coefficients of our sum at $\cos^kx,\;k\in K,$ are equal to zero. We obtain the following

\begin{theorem}\label{prince} Let $p,s\in\mathbb{N}$ and $(a_0,a_1,...,a_{p-1})\in\mathbb{R}^p$. Then for $r\geq C_1(p,s)$ there exist a vector of coefficients $(b_s,b_{s+1},...,b_r)\in\mathbb{R}^{r-s+1}$ and a polynomial $g(x)$, $\deg g=2r-2p$, such that
\begin{align}\label{cond_1}
\sum_{k=s}^{r}b_k\cos 2kx-(\cos x)^{2p} g(\cos x)\equiv \sum_{t=0}^{p-1}a_t \cos ^{2t}x
\end{align}  
and
\begin{align}\label{cond_2}
\sum_{k=s}^{r}|b_k|<\frac{C_2(p,s)}{r}\sum_{t=0}^{p-1}|a_t|,
\end{align}
where $C_1(p,s):=\max( 16p^2s^{4p-1}, 8L^{2p-1}p^3),\;L=4.56...$, and $C_2(p,s):=2^{16}p^{4p+9}s^{4p-1}$.
\end{theorem}

This means that for any $\varepsilon>0$ we can find an even polynomial with the $l_1$-norm of its coefficients less than $\varepsilon$ and with the desired first $p$ coefficients with respect to the basis $\{\cos^{2k} x\}_{k=0}^{\infty}$. 

The principle motivation for such a result is the problem of estimating the value of a trigonometric polynomial at some point $x,\;|\cos x|=\delta<1$. This can be done by rewriting the polynomial as the algebraic one and adjusting it using Theorem \ref{prince} so that its first, say, $k$ coefficients become zero. Then the value at $x$ does not exceed $\delta^k$ multiplied by the sum of absolute values of the coefficients of the obtained polynomial plus something small that comes from the adjustment. An argument of this type enables us to construct a nondegenerate double trigonometric series that converges to zero by a subsequence of squares everywhere in such a way that we can control the sizes of these squares and have explicit estimates both for the rate of convergence and for the perturbations on the intermediate steps at every point. The problem of constructing such series is closely related to that of finding universal trigonometric series (see \cite{T} and the references within).

Another application comes from the fact that equality \eqref{cond_1} can be rewritten in the form
\begin{align*}
\sum_{k=s}^{r}b_kT_{2k}(y)-y^{2p} g(y)\equiv \sum_{t=0}^{p-1}a_t y^{2t},
\end{align*}
so the result can be applied to the study of Chebyshev polynomials and Chebyshev series \cite{HM}, as series of Chebyshev polynomials are known to have properties of fast convergence among other their advantages in approximation theory and numerical analysis (see, for instance, \cite{BP}).

To prove our main result, we consider the matrix ${\bf T}=(t_m^k)_{m,k=0}^{\infty}$ whose entry $t_m^k$ is the coefficient at $x^m$ of the Chebyshev polynomial $T_k(x)$, and derive an explicit formula for the inverse of a square submatrix of $\bf{T}$. This allows us to determine the coefficients with respect to the basis $\{\cos kx\}_{k=0}^{\infty}$ of an algebraic polynomial in $\cos x$. In the course of the proof of Theorem \ref{prince}, we also give some important estimates (see Lemma \ref{str}) on sums of products of binomial coefficients appearing in the expression for entries of the pseudoinverse of a Vandermonde matrix in \cite{EPS} (see \cite{BG} for a substantive survey of generalized inverses and also \cite{PK} and \cite{V} for algebraic properties of generalized inverses of Vandermonde matrices).

\section{Inverse of a matrix containing coefficients of Chebyshev polynomials} 

Let $\bf{T_n}$ be a square $n\times n$-matrix whose entry $t_m^k$ in the $m$th row and $k$th column is the coefficient at $x^m$ of the $k$th Chebyshev polynomial $T_k(x)$ (we enumerate rows and columns of $\bf{T_n}$ beginning from $0$). It is clear that $\bf{T_n}$ is upper triangular with nonzero entries along the main diagonal. For $t_m^k$, an explicit formula is known (see, for instance, \cite[(4.5.26)]{I}):

\begin{equation*}
t_m^k = 
 \begin{cases}
   0, &\text{if $m>k$ or $k-m\equiv 1\;(\!\!\!\!\!\!\mod 2)$},\\
   (-1)^{\frac{k-m}{2}}\frac{k}{k+m}2^m\binom{\frac{k+m}{2}}{m}, &\text{otherwise.}
 \end{cases}
\end{equation*}

Denote by $ {\bf T_{k,l}}$ the $l\times l$-matrix whose entry in the $i$th row and $j$th column is equal to $t_i^{k+j}$.

\begin{lemma}\label{lem} Let $l\in\mathbb{N}$ and let $k$ be an even positive integer. The entry $g_i^j$ of the matrix $\bf{T_{k,l}^{-1}}$ is equal to $0$ if $i+j \equiv 1 \;(\!\!\!\!\mod 2)$, otherwise there hold
\begin{align*}
g_{2i}^{2j}=\frac{(-1)^{\a+j+\frac{k}{2}}(2j)!(k+i-1)!(k+2i)}{4^{j}i!(\a-i)!(\a+k+i)!}\sum\limits_{b=0}^{j}\frac{\prod\limits_{d=0,\; d\neq i}^{\a}(b^2-(\frac{k}{2}+d)^2)}{\prod\limits_{d=0,\;d\neq b}^{j}(b^2-d^2)}
\end{align*}
and
\begin{align*}
g_{2i+1}^{2j+1}=\frac{(-1)^{\beta+j+\frac{k}{2}}(2j+1)!(k+i)!}{4^{\beta}i!(\beta-i)!(\beta+k+i+1)!}\sum\limits_{b=0}^{j}\frac{\prod\limits_{d=0,\;d\neq i}^{\beta}((2b+1)^2-(k+2d+1)^2)}{\prod\limits_{d=0,\;d\neq b}^{j}((2b+1)^2-(2d+1)^2)},
\end{align*}
where $\a:=\lceil l/2\rceil -1,\;\beta:=\lfloor l/2\rfloor -1$.
\end{lemma}

\begin{proof} Note that the entries of $\bf{T_{k,l}}$ belonging to a row and a column of different parities are zero, i.e. $g_i^j=0$ for $2 \;\nmid i+j$. Fix some $j,\;0\leq j<l,$ and consider the $j$th column of $\bf{T_{k,l}^{-1}}$. Its entries must satisfy

\begin{align*}
	\sum_{i=0}^{l-1}g_i^j T_{k+i}(x)\equiv x^j+x^lg(x),
\end{align*}
where $g(x)$ is some polynomial. Rewriting this, we get

\begin{align}\label{1}
	\sum_{i=0}^{l-1}g_i^j \cos (k+i)x\equiv \cos^j x+\cos^l x\;g(\cos x).
\end{align}
We start with the case of an even $j$. Note that for any positive integer $q$ we have 
\begin{align*}
	(\cos^q x)''=(-q\sin x\cos^{q-1}x)'=q(q-1)\cos ^{q-2}x-q^2\cos^q x.
\end{align*}
So, after taking the $p$th derivative of \eqref{1} for $p=0,1,...,\lceil l/2 \rceil -1=:\alpha$, in each case we obtain at both sides polynomials in $\cos x$. As their constant terms match, we infer that

\begin{align*}
	\sum_{i=0}^{l-1}(-(k+i)^2)^pg_i^j t_0^{k+i} =y_p^j,
\end{align*} 
where $y_p^j$ stands for the constant term of $(\cos^j x )^{(2p)}$ (as of a polynomial in $\cos x$). So we have

\begin{align}\label{mat}
\begin{pmatrix}
1 & 1 & ... & 1\\
k^2 & (k+2)^2 &... & (k+2\a)^2\\
\vdots & \vdots & \vdots & \vdots\\
k^{2\a} & (k+2)^{2\a} & ... & (k+2\a)^{2\a}\\
\end{pmatrix}\text{diag}\;\big\{t_0^{k},t_0^{k+2},...,t_0^{k+2\a}\big\}
\begin{pmatrix}
g_0^j\\
g_2^j\\
\vdots\\
g_{2\a}^j
\end{pmatrix}=
\begin{pmatrix}
y_0^j\\
-y_1^j\\
\vdots\\
(-1)^{\a}y_{\a}^j\\
\end{pmatrix}.
\end{align}

Let us find $y_p^j$ for all $p$. Note that

\begin{equation}\label{y_board}
y_p^{2\a}=\begin{cases}
0,\quad p<\a,\\
(2\a)!, \quad p=\a,
\end{cases}
\end{equation}
and that 
$$\cos^j x\equiv\sum_{t=0}^{j/2}\eta_{2t}\cos 2t x.$$
The coefficients $\eta_{2t}$ can be found from the following relation:
\begin{align*}
\big(\eta_0, 0, \eta_2, 0, ..., 0, \eta_j\big)^{T}={\bf T_{j+1}^{-1}}\big(0, 0, ..., 0, 1\big)^{T}.
\end{align*}

Applying equality \eqref{mat} to the matrix $\bf{T_{j+1}^{-1}}$ and taking into account \eqref{y_board}, we derive

\begin{align}\label{eta}
\begin{pmatrix}
\eta_0\\
\eta_2\\
\vdots\\
\eta_j
\end{pmatrix}=\text{diag}\bigg\{\frac{1}{t_0^0},\frac{1}{t_0^2},...,\frac{1}{t_0^{j}}\bigg\}
\begin{pmatrix}
1 & 1 & ... & 1\\
0^2 & 2^2 &... & j^2\\
\vdots & \vdots & \vdots & \vdots\\
0^{j} & 2^j & ... & j^j\\
\end{pmatrix}^{-1}
\begin{pmatrix}
0\\
\vdots\\
0\\
(-1)^{\frac{j}{2}}j!\\
\end{pmatrix}.
\end{align}

Further,
$$(\cos^j x)^{(2p)}\equiv\sum_{t=0}^{j/2}(-4t^2)^p\eta_{2t}\cos 2t x,$$
whence in light of \eqref{eta},
\begin{align}\label{y}
&y_p^j=\sum_{m=0}^{j/2}(-4m^2)^p\eta_{2m}t_0^{2m}=\begin{pmatrix}
(-0^2)^{p}t_0^0,(-2^2)^{p}t_0^2,...,(-j^2)^{p}t_0^j
\end{pmatrix}\begin{pmatrix}
\eta_0\\
\eta_2\\
\vdots\\
\eta_j
\end{pmatrix}\nonumber\\
&=\begin{pmatrix}
(-0^2)^{p},(-2^2)^{p},...,(-j^2)^{p}
\end{pmatrix}\begin{pmatrix}
1 & 1 & ... & 1\\
0^2 & 2^2 &... & j^2\\
\vdots & \vdots & \vdots & \vdots\\
0^{j} & 2^j & ... & j^j\\
\end{pmatrix}^{-1}\begin{pmatrix}
0\\
\vdots\\
0\\
(-1)^{\frac{j}{2}}j!\\
\end{pmatrix}.
\end{align}
 
According to \cite{MS}, the $t$th element of the last column of the inverse of the Vandermonde matrix of size $m$ with the parameters $\lambda_0,...,\lambda_{m-1}$ is equal to
$$\prod_{l=0,\;l\neq t}^{m-1}(\lambda_t-\lambda_l)^{-1}.$$
Finally, we obtain
\begin{align*}
&\begin{pmatrix}
g_0^j\\
g_2^j\\
\vdots\\
g_{2\a}^j
\end{pmatrix}=(-1)^{\frac{j}{2}}j!\;\text{diag}\bigg\{\frac{1}{t_0^k},...,\frac{1}{t_0^{k+2\a}}\bigg\}\begin{pmatrix}
1 & 1 & ... & 1\\
k^2 & (k+2)^2 &... & (k+2\a)^2\\
\vdots & \vdots & \vdots & \vdots\\
k^{2\a} & (k+2)^{2\a} & ... & (k+2\a)^{2\a}\\
\end{pmatrix}^{-1}\\
&\times \begin{pmatrix}
(-0^2)^{0} & (-2^2)^{0} & ... & (-j^2)^{0}\\
-(-0^2)^{1} &-(-2^2)^{1} &... & -(-j^2)^{1}\\
\vdots & \vdots & \vdots & \vdots\\
(-1)^{\a}(-0^2)^{\a} & (-1)^{\a}(-2^2)^{\a} & ... & (-1)^{\a}(-j^2)^{\a}
\end{pmatrix}\begin{pmatrix}
\prod\limits_{t=0,\;t\neq 0}^{j/2}(-(2t)^2)^{-1}\\
\prod\limits_{t=0,\;t\neq 1}^{j/2}(2^2-(2t)^2)^{-1}\\
\vdots \\
\prod\limits_{t=0,\;t\neq j/2}^{j/2}(j^2-(2t)^2)^{-1}
\end{pmatrix}
\end{align*}
\begin{align*}
&=(-1)^{\frac{j}{2}}j!\;\text{diag}\bigg\{\frac{1}{t_0^k},...,\frac{1}{t_0^{k+2\a}}\bigg\}\begin{pmatrix}
1 & 1 & ... & 1\\
k^2 & (k+2)^2 &... & (k+2\a)^2\\
\vdots & \vdots & \vdots & \vdots\\
k^{2\a} & (k+2)^{2\a} & ... & (k+2\a)^{2\a}\\
\end{pmatrix}^{-1}\begin{pmatrix}
0^0 & 2^{0} & ... & j^0\\
0^2 & 2^2 & ... & j^2\\
\vdots & \vdots & \vdots & \vdots\\
0^{2\a} & 2^{2\a} & ... & j^{2\a}
\end{pmatrix}\\
&\times \begin{pmatrix}
\prod\limits_{t=0,\;t\neq 0}^{j/2}(-(2t)^2)^{-1}\\
\vdots \\
\prod\limits_{t=0,\;t\neq j/2}^{j/2}(j^2-(2t)^2)^{-1}
\end{pmatrix}=:(-1)^{\frac{j}{2}}j!\;\text{diag}\bigg\{\frac{1}{t_0^k},...,\frac{1}{t_0^{k+2\a}}\bigg\}{\bf J_0}\begin{pmatrix}
\prod\limits_{t=0,\;t\neq 0}^{j/2}(-(2t)^2)^{-1}\\
\vdots \\
\prod\limits_{t=0,\;t\neq j/2}^{j/2}(j^2-(2t)^2)^{-1}
\end{pmatrix}.
\end{align*}
The matrix ${\bf J_0}$ is of size $(\a+1) \times (j/2+1)$ and its entries are
\begin{align*}
j_{a}^b=\frac{\prod\limits_{d=0,\;d\neq a}^{\a}((2b)^2-(k+2d)^2)}{\prod\limits_{d=0,\;d\neq a}^{\a}((k+2a)^2-(k+2d)^2)},
\end{align*}
since the entry $v_i^j$ of the square Vandermonde matrix with the parameters $k^2,(k+2)^2,...,(k+2\a)^2$ is equal to 
\begin{align*}
v_i^j=\frac{\Big[\prod\limits_{d=0,\;d\neq i}^{\a}(x-(k+2d)^2)\Big]_{j}}{\prod\limits_{d=0,\;d\neq i}^{\a}((k+2i)^2-(k+2d)^2)},
\end{align*}
where $[P(x)]_j$ stands for the coefficient at $x^j$ of the polynomial $P(x)$. Hence, recalling that $\a=\lceil l/2\rceil -1$, we have
\begin{align}\label{hence1}
g_{2a}^j=\frac{(-1)^{\frac{j}{2}+a+\frac{k}{2}}j!}{\prod\limits_{d=0,\;d\neq a}^{\lceil l/2\rceil-1}((k+2a)^2-(k+2d)^2)}\sum_{b=0}^{j/2}\frac{\prod\limits_{d=0,\;d\neq a}^{\lceil l/2\rceil-1}((2b)^2-(k+2d)^2)}{\prod\limits_{d=0,\;d\neq b}^{j/2}((2b)^2-(2d)^2)}.
\end{align}

Turn now to the case of an odd $j$. Once more, taking the $p$th derivative of \eqref{1} for $p=0,1,...,\lfloor l/2 \rfloor -1=:\beta$ and obtaining the same coefficients at $\cos x$, we get

\begin{align*}
	\sum_{i=0}^{l-1}(-(k+i)^2)^pg_i^j t_1^{k+i} =z_p^j,
\end{align*} 
where $z_p^j$ is the coefficient at $\cos x$ of $(\cos^{j+1} x )^{(2p)}$ (as of a polynomial in $\cos x$). We have

\begin{align}\label{mat1}
\begin{pmatrix}
1 & 1 & ... & 1\\
(k+1)^2 & (k+3)^2 &... & (k+1+2\beta)^2\\
\vdots & \vdots & \vdots & \vdots\\
(k+1)^{2\beta} & (k+3)^{2\beta} & ... & (k+1+2\beta)^{2\beta}\\
\end{pmatrix}\text{diag}\;(t_1^{k+1},...,t_1^{k+1+2\beta})
\begin{pmatrix}
g_1^j\\
g_3^j\\
\vdots\\
g_{2\beta+1}^j
\end{pmatrix}=
\begin{pmatrix}
z_0^j\\
-z_1^j\\
\vdots\\
(-1)^{\beta}z_{\beta}^j\\
\end{pmatrix}.
\end{align}
Noting that

\begin{equation}\label{y_board1}
z_p^{2\beta}=\begin{cases}
0,\quad p<\beta,\\
(2\beta+1)!, \quad p=\beta,
\end{cases}
\end{equation}
and that 
$$\cos^j x\equiv\sum_{t=0}^{(j-1)/2}\eta_{2t+1}\cos (2t+1) x,$$
we derive
\begin{align*}
\big(0,\eta_1, 0, \eta_3, ..., 0, \eta_j\big)^{T}={\bf T_{j+1}^{-1}}\big(0, 0, ..., 0, 1\big)^{T}.
\end{align*}
Applying \eqref{mat1} to $\bf{T_{j+1}^{-1}}$ and using \eqref{y_board1}, we obtain

\begin{align*}
\begin{pmatrix}
\eta_1\\
\eta_3\\
\vdots\\
\eta_j
\end{pmatrix}=\text{diag}\bigg\{\frac{1}{t_1^1},...,\frac{1}{t_1^j}\bigg\}
\begin{pmatrix}
1 & 1 & ... & 1\\
1^2 & 3^2 &... & j^2\\
\vdots & \vdots & \vdots & \vdots\\
1^{j} & 3^j & ... & j^j\\
\end{pmatrix}^{-1}
\begin{pmatrix}
0\\
\vdots\\
0\\
(-1)^{\frac{j-1}{2}}j!\\
\end{pmatrix}.
\end{align*}
Further,
$$(\cos^j x)^{(2p)}\equiv\sum_{t=0}^{j/2}(-(2t+1)^2)^p\eta_{2t+1}\cos (2t+1) x,$$
whence 
\begin{align*}
&z_p^j=\sum_{m=0}^{\frac{j-1}{2}}(-(2m+1)^2)^p\eta_{2m+1}t_1^{2m+1}=\begin{pmatrix}
(-1^2)^{p}t_1^1,(-3^2)^{p}t_1^3,...,(-j^2)^{p}t_1^j
\end{pmatrix}\begin{pmatrix}
\eta_1\\
\eta_3\\
\vdots\\
\eta_j
\end{pmatrix}\nonumber\\
&=\begin{pmatrix}
(-1^2)^{p},(-3^2)^{p},...,(-j^2)^{p}
\end{pmatrix}\begin{pmatrix}
1 & 1 & ... & 1\\
1^2 & 3^2 &... & j^2\\
\vdots & \vdots & \vdots & \vdots\\
1^{j} & 3^j & ... & j^j\\
\end{pmatrix}^{-1}\begin{pmatrix}
0\\
\vdots\\
0\\
(-1)^{\frac{j-1}{2}}j!\\
\end{pmatrix}.
\end{align*}
 
Finally, as before
\begin{align*}
&\begin{pmatrix}
g_1^j\\
g_3^j\\
\vdots\\
g_{2\beta+1}^j
\end{pmatrix}=(-1)^{\frac{j-1}{2}}j!\;\text{diag}\bigg\{\frac{1}{t_1^{k+1}},...,\frac{1}{t_1^{k+1+2\beta}}\bigg\}
\begin{pmatrix}
1 & 1 & ... & 1\\
(k+1)^2 & (k+3)^2 &... & (k+1+2\beta)^2\\
\vdots & \vdots & \vdots & \vdots\\
(k+1)^{2\beta} & (k+3)^{2\beta} & ... & (k+1+2\beta)^{2\beta}\\
\end{pmatrix}^{-1}\\
&\times \begin{pmatrix}
1 & 1 & ... & 1\\
1^2 & 3^2 &... & j^2\\
\vdots & \vdots & \vdots & \vdots\\
1^{2\beta} & 3^{2\beta} & ... & j^{2\beta}
\end{pmatrix}\begin{pmatrix}
\prod\limits_{t=0,\;t\neq 0}^{(j-1)/2}(1^2-(2t+1)^2)^{-1}\\
\prod\limits_{t=0,\;t\neq 1}^{(j-1)/2}(3^2-(2t+1)^2)^{-1}\\
\vdots \\
\prod\limits_{t=0,\;t\neq (j-1)/2}^{(j-1)/2}(j^2-(2t+1)^2)^{-1}
\end{pmatrix}\\
&=:(-1)^{\frac{j-1}{2}}j!\;\text{diag}\bigg\{\frac{1}{t_1^{k+1}},...,\frac{1}{t_1^{k+1+2\beta}}\bigg\}{\bf J_1}\begin{pmatrix}
\prod\limits_{t=0,\;t\neq 0}^{(j-1)/2}(1^2-(2t+1)^2)^{-1}\\
\vdots \\
\prod\limits_{t=0,\;t\neq (j-1)/2}^{(j-1)/2}(j^2-(2t+1)^2)^{-1}
\end{pmatrix}.
\end{align*}
The matrix ${\bf J_1}$ is of size $(\beta+1) \times ((j-1)/2+1)$ and its entries are
\begin{align*}
j_{a}^b=\frac{\prod\limits_{d=0,\;d\neq a}^{\beta}((2b+1)^2-(k+2d+1)^2)}{\prod\limits_{d=0,\;d\neq a}^{\beta}((k+2a+1)^2-(k+2d+1)^2)}.
\end{align*}
Hence, 

\begin{align}\label{hence2}
g_{2a+1}^j&=\frac{(-1)^{\frac{j-1}{2}+a+\frac{k}{2}}j!}{(k+2a+1)\prod\limits_{d=0,\;d\neq a}^{\lfloor l/2\rfloor-1}((k+2a+1)^2-(k+2d+1)^2)}\sum_{b=0}^{(j-1)/2}\frac{\prod\limits_{d=0,\;d\neq a}^{\lfloor l/2\rfloor-1}((2b+1)^2-(k+2d+1)^2)}{\prod\limits_{d=0,\;d\neq b}^{(j-1)/2}((2b+1)^2-(2d+1)^2)},
\end{align}
and the claim follows.
\end{proof}

\begin{remark} Following the ideas of the proof of Lemma \ref{lem}, one can establish an explicit formula for the elements of the inverse of any submatrix ${\bf T_{k,l,m}}:=(t_{m+i}^{k+j})_{i,j=0}^{l-1}$ of ${\bf T_n}$ with even $k$ and $m$.
\end{remark}

\begin{remark}\label{lemc} For any $n\in\mathbb{N}$, the entry $h_i^j$ of the matrix $\bf{T_n^{-1}}$ is zero if $2\;\nmid i+j$ or $i>j$, otherwise can be found by
\begin{align*}
	h_{2i}^{2j}=2^{\delta_{i}-2j}\binom{2j}{j-i},\qquad h_{2i+1}^{2j+1}=2^{\delta_{i}-2j}\binom{2j+1}{j-i},
\end{align*}
where 
\begin{equation*}
\delta_i:=\begin{cases} 0,\quad\text{if}\;\; i=0,\\
1,\quad\text{if}\;\; i\neq 0.
\end{cases}
\end{equation*}
\end{remark}

\begin{proof} Noting that, for $b\neq a$, 
$$\prod\limits_{d=0,\;d\neq a}^{\lceil n/2\rceil-1}((2b)^2-(2d)^2)=0,$$
we obtain $h_{2i}^{2j}=0$, for $i>j$, and otherwise due to \eqref{hence1},
\begin{align*}
h_{0}^{2j}=\frac{(2j)!}{((2j)!!)^2}=2^{-2j}\binom{2j}{j},
\end{align*}
and 
\begin{align*}
h_{2i}^{2j}=\frac{(-1)^{j+i}(2j)!}{\prod\limits_{d=0,\;d\neq i}^{\lceil n/2\rceil-1}((2i)^2-(2d)^2)}\frac{\prod\limits_{d=0,\;d\neq i}^{\lceil n/2\rceil-1}((2i)^2-(2d)^2)}{\prod\limits_{d=0,\;d\neq i}^{j}((2i)^2-(2d)^2)}&=\frac{(-1)^{j+i}(2j)!}{(2i)!!(-1)^{j-i}(2j-2i)!!\frac{(2j+2i)!!}{(2i-2)!!4i}}\\
& =\frac{2^{1-2j}(2j)!}{(j-i)!(j+i)!}=2^{1-2j}\binom{2j}{j-i},
\end{align*}
if $i>0$.

For odd entries, once more we get $h_{2i+1}^{2j+1}=0$ for $i>j$, otherwise from \eqref{hence2},
\begin{align*}
h_{1}^{2j+1}=\frac{(2j+1)!}{(2j+2)!!(2j)!!}=2^{-2j}\binom{2j+1}{j},
\end{align*}
and
\begin{align*}
h_{2i+1}^j&=\frac{(-1)^{j+i}(2j+1)!(-1)^{j-i}(2i)!!(4i+2)}{(2i+1)(2i)!!(2j-2i)!!(2i+2j+2)!!}=2^{1-2j}\binom{2j+1}{j-i},
\end{align*}
if $i>0$, so the proof is complete.
\end{proof}

\begin{corollary}\label{cor} There holds

\begin{align*}
(\cos^{2j})^{(2p)}|_{x=\pi/2}=:y^{2j}_p=(-4)^{p-j}\sum_{k=0}^{2j}(-1)^k\binom{2j}{k}(j-k)^{2p}.
\end{align*}
\end{corollary}

\begin{proof}
It follows from \eqref{y} that
\begin{align*}
y^{2j}_p&=(-1)^{j}(2j)!\begin{pmatrix}
(-0^2)^p, & (-2^2)^p, & ... ,&(-(2j)^2)^p
\end{pmatrix} \begin{pmatrix}
\prod\limits_{t=0,\;t\neq 0}^{j}(-(2t)^2)^{-1}\\
\prod\limits_{t=0,\;t\neq 1}^{j}(2^2-(2t)^2)^{-1}\\
\vdots \\
\prod\limits_{t=0,\;t\neq j}^{j}((2j)^2-(2t)^2)^{-1}
\end{pmatrix}\\
&=(-1)^{p+j}(2j)!\sum_{a=0}^{j}\frac{(2a)^{2p}}{\prod\limits_{t=0,\;t\neq a}^{j}((2a)^2-(2t)^2)}=(-4)^p2^{-2j}\sum_{a=0}^j(-1)^a\binom{2j}{j-a}a^{2p}2^{\delta_a}\\
&=(-4)^{p-j}\sum_{k=0}^{2j}(-1)^k\binom{2j}{k}(j-k)^{2p},
\end{align*}
where $\delta_a$ is as in Remark \ref{lemc}, and we are done.
\end{proof}

\section{Proof of Theorem \ref{prince}}
Now we are ready to prove the main theorem.

\begin{proof}[Proof of Theorem \ref{prince}.] For the sake of clarity, let us split the proof into three main parts.
\subsection{Finding a sufficient condition for \eqref{cond_1} to hold.}
First we note that \eqref{cond_1} is equivalent to

\begin{align*}
\begin{pmatrix}
t_0^{2s} & t_0^{2s+2} & ... & t_0^{2r}\\
t_2^{2s} & t_2^{2s+2} & ... & t_2^{2r}\\
\vdots & \vdots & \vdots & \vdots \\
t_{2p-2}^{2s} & t_{2p-2}^{2s+2} & ... & t_{2p-2}^{2r}
\end{pmatrix}\begin{pmatrix}
b_s\\
b_{s+1}\\
\vdots\\
b_{r}
\end{pmatrix}=\begin{pmatrix}
a_0\\
a_1\\
\vdots\\
a_{p-1}
\end{pmatrix}.
\end{align*}

Pick some $k\in\{s,s+1,...,r\}$ and take the $(2q)$th derivative of the equality

\begin{align*}
\cos 2kx\equiv\sum_{l=0}^k t_{2l}^{2k}\cos^{2l}x,
\end{align*}
where $q\in\{0,1,...,p-1\}$, at the point $\pi/2$. What we get is

\begin{align*}
(-1)^q (2k)^{2q}t_0^{2k}=\sum_{l=0}^{k} t_{2l}^{2k}(\cos ^{2l}x)^{(2q)}|_{x=\pi/2},
\end{align*}
which is equivalent to
\begin{align*}
(-1)^{q+k} (2k)^{2q}=\sum_{l=0}^{k} t_{2l}^{2k}y^{2l}_q.
\end{align*}

From the relations above for all $k$ and $q$ in the mentioned ranges, we derive

\begin{align*}
\begin{pmatrix}
y_0^{0} & y_0^{2} & ... & y_0^{2p-2}\\
y_1^{0} & y_1^{2} & ... & y_1^{2p-2}\\
\vdots & \vdots & \vdots & \vdots\\
y_{p-1}^{0} & y_{p-1}^{2} & ... & y_{p-1}^{2p-2}\\
\end{pmatrix}\begin{pmatrix}
t_0^{2s} & t_0^{2s+2} & ... & t_0^{2r}\\
t_2^{2s} & t_2^{2s+2} & ... & t_2^{2r}\\
\vdots & \vdots & \vdots & \vdots \\
t_{2p-2}^{2s} & t_{2p-2}^{2s+2} & ... & t_{2p-2}^{2r}
\end{pmatrix}=N_p\begin{pmatrix}
(2s)^0 & (2s+2)^0 & ... & (2r)^0\\
(2s)^2 & (2s+2)^2 & ... & (2r)^2\\
\vdots & \vdots & \vdots & \vdots \\
(2s)^{2p-2} & (2s+2)^{2p-2} & ... & (2r)^{2p-2}
\end{pmatrix}N_{r-s+1},
\end{align*}
where $N_j$ is a square diagonal matrix of size $j$ with its entries belonging to the even columns equal $1$, while the entries belonging to the odd ones, equal $-1$. Thus,

\begin{align}\label{matr_sootn}
\begin{pmatrix}
t_0^{2s} & t_0^{2s+2} & ... & t_0^{2r}\\
t_2^{2s} & t_2^{2s+2} & ... & t_2^{2r}\\
\vdots & \vdots & \vdots & \vdots \\
t_{2p-2}^{2s} & t_{2p-2}^{2s+2} & ... & t_{2p-2}^{2r}
\end{pmatrix}=\begin{pmatrix}
y_0^{0} & y_0^{2} & ... & y_0^{2p-2}\\
\frac{y_1^{0}}{4} & \frac{y_1^{2}}{4} & ... & \frac{y_1^{2p-2}}{4}\\
\vdots & \vdots & \vdots & \vdots\\
\frac{y_{p-1}^{0}}{2^{2p-2}} & \frac{y_{p-1}^{2}}{2^{2p-2}} & ... & \frac{y_{p-1}^{2p-2}}{2^{2p-2}}
\end{pmatrix}^{-1}N_p\begin{pmatrix}
s^0 & (s+1)^0 & ... & r^0\\
s^2 & (s+1)^2 & ... & r^2\\
\vdots & \vdots & \vdots & \vdots \\
s^{2p-2} & (s+1)^{2p-2} & ... & r^{2p-2}
\end{pmatrix}N_{r-s+1}.
\end{align}

Let $Y$ be a square matrix of size $2p$ such that the matrix generated by the odd rows and the odd columns of $Y$ (as before, we start enumerating from zero) is the identity matrix, an entry belonging to the $(2u)$th column and $(2k)$th row is equal to $2^{-2k}y_k^{2u}$, the other entries are zeros. Then $Y$ is invertible due to invertibility of the identity matrix and that of the matrix $(2^{-2i}y_i^{2j})_{i,j=0}^{p-1}$. Therefore, it follows from \eqref{matr_sootn} that there exist $\tau_{2i+1}^{2j},\;i=0,...,p-1,\; j=s,s+1,...,r$, such that
 
\begin{align}\label{matr_sootn2}
T:=\begin{pmatrix}
t_0^{2s} & t_0^{2s+2} & ... & t_0^{2r}\\
\tau_1^{2s} & \tau_1^{2s+2} & ... & \tau_1^{2r}\\
t_2^{2s} & t_2^{2s+2} & ... & t_2^{2r}\\
\tau_3^{2s} & \tau_3^{2s+2} & ... & \tau_3^{2r}\\
\vdots & \vdots & \vdots & \vdots \\
t_{2p-2}^{2s} & t_{2p-2}^{2s+2} & ... & t_{2p-2}^{2r}\\
\tau_{2p-1}^{2s} & \tau_{2p-1}^{2s+2} & ... & \tau_{2p-1}^{2r}
\end{pmatrix}&=Y^{-1}\tilde{N}_p\begin{pmatrix}
s^0 & (s+1)^0 & ... & r^0\\
s^1 & (s+1)^1 & ... & r^1\\
s^2 & (s+1)^2 & ... & r^2\\
s^3 & (s+1)^3 & ... & r^3\\
\vdots & \vdots & \vdots & \vdots \\
s^{2p-2} & (s+1)^{2p-2} & ... & r^{2p-2}\\
s^{2p-1} & (s+1)^{2p-1} & ... & r^{2p-1}
\end{pmatrix}N_{r-s-1}\nonumber\\
&=:Y^{-1}\tilde{N}_pV N_{r-s+1}.
\end{align}
Here $\tilde{N}_j$ stands for a square matrix of size $2j$ having $N_j$ in the intersection of the even columns and the even rows, $E_j$ in the intersection of the odd columns and the odd rows, and the other entries equal zero.

Note that if we make ${\bf b}:=(b_s,...,b_r)^T$ satisfy the equality 

\begin{align*}
T{\bf b}={\bf a},
\end{align*}
where ${\bf a}:=(a_0,0,a_1,0,...,a_{p-1},0)^T$, then condition \eqref{cond_1} will be fulfilled. 

\subsection{Constructing a vector of coefficients.}
Let

\begin{align*}
{\bf b}:=T^*(TT^*)^{-1}{\bf a}.
\end{align*}
Since using \eqref{matr_sootn2} we have
$$T^{\dagger}:=(TT^*)^{-1}T=(Y^{-1}\tilde{N}_pVN_{r-s+1}N_{r-s+1}^*V^*\tilde{N}_p^*(Y^{-1})^*)^{-1}Y^{-1}\tilde{N}_pVN_{r-s+1}=Y^*\tilde{N}_p(VV^*)^{-1}VN_{r-s+1},$$ 
the definition of ${\bf b}$ is equivalent to

\begin{align}\label{dag}
\begin{pmatrix}
b_s, ..., b_r\end{pmatrix}&=\begin{pmatrix}
a_0, 0, ..., a_{p-1},0
\end{pmatrix}Y^*\tilde{N}_p(VV^*)^{-1}VN_{r-s+1}=:\begin{pmatrix}
a_0, 0, ..., a_{p-1},0
\end{pmatrix}Y^*\tilde{N}_pV^{\dagger}N_{r-s+1},
\end{align}
where $V^{\dagger}$ is the pseudoinverse for the Vandermonde matrix $V$. Note that

\begin{align*}
VV^*=WW^*-ZZ^*,
\end{align*}
where $W=(w_i^j),\;w_i^j:=(j+1)^i,\;j=0,...,r-1,i=0,...,2p-1,\;Z=(z_i^j),\;z_i^j:=(j+1)^i,\;j=0,...,s-2,i=0,...,2p-1.$ According to \cite[(10)]{EPS}, the condition number of $WW^*$ is

\begin{align*}
\hat{\kappa}=\frac{(2p)^2}{4p-1}r^{4p-2}.
\end{align*}

The maximal entry of $WW^*$ is greater than $r^{4p-1}/(4p-1)$, therefore the $l^2$-norm of this matrix exceeds this value. Thus, $\|(WW^*)^{-1}\|_2<\hat{\kappa}(4p-1)/r^{4p-1}<8p^2/r$. In turn, $\|ZZ^*\|_2< (s-1)^{4p-1}$, which yields $\|(WW^*)^{-1}\|_2\|ZZ^*\|_2\leq 8p^{2}s^{4p-1}/r\leq 0.5$. So, we have the following representation

\begin{align}\label{WW}
(VV^*)^{-1}=(WW^*-ZZ^*)^{-1}=\sum_{k=0}^{\infty}((WW^*)^{-1}ZZ^*)^k(WW^*)^{-1}=:(E_{2p}+X)(WW^*)^{-1},
\end{align}
where $E_{2p}$ is the identity matrix of size $2p$ and 

\begin{align}\label{X}
\|X\|_2<\frac{8p^2s^{4p-1}}{r-8p^2s^{4p-1}}.
\end{align}
 
Due to \cite[relation before Prop. 3]{EPS}, for entries of $W^{\dagger}=(WW^*)^{-1}W$ we have (taking into account that we enumerate from zero)

\begin{align}\label{W}
(W^{\dagger})_{q,k}=(-1)^q\sum_{w=q}^{2p-1}\frac{1}{w!}s(w+1,q+1)\sum_{t=w}^{2p-1}\frac{\binom{t+w}{w}\binom{r-w-1}{r-t-1}}{\binom{2t}{t}\binom{r+t}{2t+1}}\sum_{j=0}^{\min(t,k)} (-1)^{j+1}\binom{k}{j}\binom{j+t}{j}\binom{r-j-1}{r-t-1},
\end{align}
where $s(w+1,q+1)$ is the Stirling number of the first kind. Further, we have from Corollary \ref{cor}

\begin{align}\label{yy}
y_q^{2u}=(-4)^{q-u}\sum_{v=0}^{2u}(-1)^v\binom{2u}{v}(u-v)^{2q},
\end{align}
hence, there holds

\begin{align}\label{uff}
&(Y^*\tilde{N}_{p}W^{\dagger})_{2u,k}=(-4)^{-u}\sum_{v=0}^{2u} (-1)^v \binom{2u}{v} \sum_{w=0}^{2p-1}\frac{1}{w!}\sum_{q=0}^{\lfloor w/2\rfloor}(-1)^q (u-v)^{2q} s(w+1,2q+1)\nonumber\\
&\qquad\qquad\qquad\cdot\sum_{t=w}^{2p-1}\frac{\binom{t+w}{w}\binom{r-w-1}{r-t-1}}{\binom{2t}{t}\binom{r+t}{2t+1}}\sum_{j=0}^{\min(t,k)} (-1)^{j+1}\binom{k}{j}\binom{j+t}{j}\binom{r-j-1}{r-t-1}.
\end{align}

\subsection{Estimating $T^{\dagger}$.}
First, by the definition of Stirling numbers, $s(w+1,2q+1)$ is the coefficient at $x^{2q+1}$ of the polynomial $x(x+1)...(x+w)$, which is the same as to be the coefficient at $x^{2q}$ of the polynomial $(x+1)...(x+w)$. Thus,

\begin{align}\label{Ree}
\sum_{q=0}^{\lfloor w/2\rfloor}s(w+1,2q+1)(-1)^q(u-v)^{2q}=\Re (i(u-v)+1)...(i(u-v)+w)<\frac{(|u-v|+w)!}{|u-v|!}.
\end{align}

To obtain upper bounds for the sum
\begin{align*}
A(t,k,r):=\sum_{j=0}^{\min (k,t)}(-1)^j\binom{k}{j}\binom{j+t}{j}\binom{r-j-1}{r-t-1}
\end{align*}
that appears in \eqref{uff}, we need the following
\begin{lemma}\label{str} Let $t,q,k,r\in\mathbb{N}$ be such that $q\geq t\geq 2$. 
\\
a) If $r\geq q+2q^2$ and $r-q-1\geq k\geq q$, then 
\begin{align}\label{str1}
|A(t,k,r)|=\left |\sum_{j=0}^{\min (k,t)}(-1)^j\binom{k}{j}\binom{j+t}{j}\binom{r-j-1}{r-t-1}\right|\leq 4\left(\frac{q}{r-1-q}\right)^{q-t}\binom{r-1}{t}.
\end{align}
b) If $r\geq 2t^3+t$ and $k<t$, then
\begin{align}\label{str2}
|A(t,k,r)|< \binom{r-1}{t}.
\end{align}
c) If $r\geq 2L^t t^{1.5},$ where $L:=(\sqrt{2}+1)^{1+\frac{1}{\sqrt{2}}}(\sqrt{2}-1)^{-1+\frac{1}{\sqrt{2}}}2^{-\frac{1}{2\sqrt{2}}}$, and $t\leq k\leq r-1$, then 
\begin{align}\label{str3}
|A(t,k,r)|< 3\binom{r-1}{t}.
\end{align}
\end{lemma}

\begin{proof} a) We begin with estimate \eqref{str1} and the proof will be divided into several steps. 
\\

{\bf Step 1. Algebraic representation of $A(t,k,r)$.} It turns out that the sum of products of binomial coefficients $A(t,k,r)$ has the following algebraic meaning: it represents the coefficient at $x^t y^k$ of the Taylor expansion at zero of the function
\begin{align*}
G(x,y):=\frac{(1+y)^k}{(1+xy)^{t+1}(1-x)^{r-t}}.
\end{align*}
Indeed,
\begin{align*}
\binom{r-j-1}{r-t-1}&=\binom{r-j-1}{t-j}=\frac{(r-j-1)(r-j-2)...(r-t)}{(t-j)!}\\
&=(-1)^{t-j}\frac{(-(r-t))(-(r-t+1))...(-(r-j-1))}{(t-j)!}=(-1)^{t-j}\binom{-r+t}{t-j},
\end{align*}
so we have for $k\geq q\geq t$
\begin{align*}
A(t,k,r)=\sum_{j=0}^{\min (k,t)}(-1)^j\binom{k}{j}\binom{j+t}{j}\binom{r-j-1}{r-t-1}=\sum_{j=0}^{\min (k,t)}\binom{k}{j}\binom{-t-1}{j}(-1)^{t-j}\binom{-r+t}{t-j},
\end{align*}
which corresponds to the mentioned coefficient.

Take some $\varepsilon\in(0,1)$ and let $\delta=t/r$. By Cauchy's formulas,

\begin{align*}
&A(t,k,r)=\frac{1}{(2\pi i)^2}\int\limits_{|y|=\varepsilon}\int\limits_{|x|=\varepsilon}G(x,y)x^{-t-1}y^{-k-1}dx \;dy=\frac{1}{(2\pi i)^2}\int\limits_{|y|=\varepsilon}\frac{(1+y)^k}{y^{k+1}}\\
\cdot &\left(\int\limits_{|x|=\varepsilon^{\delta-1}}\frac{1}{(1+xy)^{t+1}(1-x)^{r-t}x^{t+1}}dx\; -2\pi i \;\text{res}_{x=1}\frac{1}{(1+xy)^{t+1}(1-x)^{r-t}x^{t+1}}\right)dy =:S_1+S_2.
\end{align*} 

{\bf Step 2. Estimating $S_2$.} We have
\begin{align*}
&\text{res}_{x=1}\frac{1}{(1+xy)^{t+1}(1-x)^{r-t}x^{t+1}}=\frac{(-1)^{r-t}}{(r-t-1)!}\left(\frac{1}{(1+xy)^{t+1}x^{t+1}}\right)^{(r-t-1)}\bigg|_{x=1}\\
&=\frac{(-1)^{r-t}}{(r-t-1)!}\sum_{l=0}^{r-t-1}\left(\frac{1}{x^{t+1}}\right)^{(l)}\left(\frac{1}{(1+xy)^{t+1}}\right)^{(r-t-1-l)}\bigg|_{x=1}\\
&=\frac{-1}{(r-t-1)!}\sum_{l=0}^{r-t-1}\frac{(t+l)!}{t!}\frac{(r-l-1)!}{t!}\frac{y^{r-t-1-l}}{(1+y)^{r-l}},
\end{align*}
hence, $S_2$ is the coefficient at $y^0$ of the Laurent expansion of the function
\begin{align*}
&\frac{(1+y)^k}{y^{k}}\frac{-1}{(r-t-1)!}\sum_{l=0}^{r-t-1}\frac{(t+l)!}{t!}\frac{(r-l-1)!}{t!}\frac{y^{r-t-1-l}}{(1+y)^{r-l}}\\
=&\frac{-1}{(r-t-1)!(t!)^2}\sum_{l=0}^{r-t-1}(t+l)!(r-l-1)!y^{r-t-1-l-k}\sum_{j=0}^{\infty}\binom{-r+l+k}{j}y^j.
\end{align*}
Note that for $l$ satisfying $r-t-1-l-k>0$, the coefficient of the corresponding term at $y^0$ is zero, therefore it suffices to consider just $l\geq r-t-1-k$. At the same time, if $-r+l+k\geq 0$, then $-(r-t-1-l-k)=-r+t+1+l+k>-r+l+k$, so $\binom{-r+l+k}{-(r-t-1-l-k)}=0$, which means that for $l\geq r-k$ the corresponding term is zero. Hence,

\begin{align*}
S_2&=\frac{-1}{(r-t-1)!(t!)^2}\sum_{l=r-t-1-k}^{r-k-1}(t+l)!(r-l-1)!\binom{-r+l+k}{-r+t+1+l+k}\\
&=-\sum_{l=r-t-1-k}^{r-k-1}(-1)^{r-t-1-l-k}\frac{(t+l)!(r-l-1)!(r-l-k-1-r+t+1+l+k)!}{(r-t-1)!(t!)^2(-r+t+1+l+k)!(r-l-k-1)!}\\
&=-\sum_{m=0}^{t}(-1)^m\frac{(r-1-k+m)!(t+k-m)!}{(r-t-1)!t!m!(t-m)!}=:-\sum_{m=0}^{t}(-1)^m D_m(r,t,k).
\end{align*}
{\bf Step 2.1. Estimating $D_m$.} Note that
\begin{align*}
\frac{D_m(r,t,k+1)}{D_m(r,t,k)}=\frac{t+k+1-m}{r-1-k+m},
\end{align*}
and since $q<\frac{r-t}{2}-1+m<r-q-1$, the maximum of the above expression is attained either at $k=q$ or at $k=r-1-q$.

For $k=q$, we have
\begin{align*}
\frac{D_{m+1}(r,t,q)}{D_m(r,t,q)}=\frac{(r-q+m)(t-m)}{(m+1)(t+q-m)}=\frac{rt-qt+mt-rm+qm-m^2}{mt+t+qm+q-m^2-m}>1,
\end{align*}
since
\begin{align*}
rt-qt-rm\geq r-qt\geq t+q\geq t+q-m.
\end{align*}
Thus, $D_m(r,t,q)$ is maximal at $m=t$ and 
$$D_t(r,t,q)=\frac{(r-1-q+t)!q!}{(r-t-1)!(t!)^2}\leq\left(\frac{q}{r-1-q}\right)^{q-t}\binom{r-1}{t}.$$
For $k=r-1-q$, we have
\begin{align*}
\frac{D_{m+1}(r,t,r-1-q)}{D_m(r,t,r-1-q)}=\frac{(q+m+1)(t-m)}{(m+1)(r-1+t-q-m)}=\frac{qt+mt+t-qm-m^2-m}{rm+r-2m-1+tm+t-qm-q-m^2}<1,
\end{align*}
in light of
\begin{align*}
rm+r-m-1+t-q\geq r-1-q\geq qt+t.
\end{align*}
Thus, $D_m(r,t,r-1-q)$ is maximal at $m=0$ and $$D_0(r,t,r-1-q)=\frac{q!(r-1+t-q)!}{(r-t-1)!(t!)^2}\leq\left(\frac{q}{r-1-q}\right)^{q-t}\binom{r-1}{t}.$$

Finally,
\begin{align*}
|S_2|\leq 4\left(\frac{q}{r-1-q}\right)^{q-t}\binom{r-1}{t}.
\end{align*}

{\bf Step 3. Estimating $S_1$.} Observe that for small enough $\varepsilon$ and $|y|=\varepsilon, \;|x|=\varepsilon^{\delta-1}$, the function $|G(x,y)x^{-t-1}y^{-k-1}|$ is equivalent to
\begin{align*}
\varepsilon^{-(\delta-1)(r-t)}\varepsilon^{-k-1}\varepsilon^{-(\delta-1)(t+1)},
\end{align*}
then
\begin{align*}
|S_1|\lesssim \varepsilon\varepsilon^{\delta-1}\varepsilon^{-(\delta-1)(r-t)-k-1-(\delta-1)(t+1)}=\varepsilon^{r-t-k}\leq \varepsilon^{q+1-t}\underset{\varepsilon\to 0}{\to}0,
\end{align*}
whence we get \eqref{str1}.

b) Turn now to \eqref{str2} for $r\geq 2t^3+t$ and $k<t$. We have

\begin{align*}
A(t,k,r)&=\sum_{j=0}^{\min(k,t)} (-1)^{j}\frac{k!}{j!(k-j)!}\frac{(j+t)!}{j!t!}\frac{(r-j-1)!}{(t-j)!(r-t-1)!}\\
&=\frac{(r-1)!}{t!(r-t-1)!}\sum_{j=0}^t(-1)^{j}\binom{t}{j}\binom{t+j}{j}\frac{k(k-1)...(k-j+1)}{(r-1)(r-2)...(r-j)}.
\end{align*}
For all $j$, there holds
\begin{align*}
\binom{t}{j}\binom{t+j}{j}\frac{k(k-1)...(k-j+1)}{(r-1)(r-2)...(r-j)}< t^j(2t)^j\frac{t^j}{(r-t)^j}\leq 1,
\end{align*}
since $r\geq 2t^3+t$. Note that the expression above decreases. Indeed, going from $j$ to $j+1$ we get our value changed in 
$$\frac{(j+t+1)(k-j)(t-j)}{(r-j-1)(j+1)^2}<\frac{2t\cdot t^2}{r-t}\leq 1$$
times. Therefore, we derive 
\begin{align*}
A(t,k,r)< \binom{r-1}{t}.
\end{align*}

c) Now we have only to prove \eqref{str3} under the mentioned conditions. Divide our sum into two sums in the following way:

\begin{align*}
A(t,k,r)&=\binom{r-1}{t}\sum_{j=0}^t(-1)^j\binom{t}{j}\binom{t+j}{j}\Big(\frac{k+1}{r}\Big)^j\\
&+\binom{r-1}{t}\sum_{j=1}^t(-1)^{j}\binom{t}{j}\binom{t+j}{j}\left(\frac{k...(k-j+1)}{(r-1)...(r-j)}-\Big(\frac{k+1}{r}\Big)^j\right)=:\binom{r-1}{t}(S_3+S_4).
\end{align*}
{\bf Step 1. Estimating $S_3$.} Since $$\binom{t+j}{j}=\frac{(t+j)(t+j-1)...(t+1)}{j!}=(-1)^j\frac{(-t-1)(-t-2)...(-t-j)}{j!}=(-1)^j\binom{-t-1}{j},$$
we have
\begin{align*}
S_3=\sum_{j=0}^t \binom{t}{t-g}\binom{-t-1}{j}\Big(\frac{k+1}{r}\Big)^j=\frac{1}{t!}\Big((1+x)^t\Big(1+\Big(\frac{k+1}{r}\Big) x\Big)^{-t-1}\Big)^{(t)}\big|_{x=0}.
\end{align*}
To estimate this value, we will need the following 

\begin{lemma}\label{g} For any positive integer $n$ and any $\g\in (0,1)$, there holds

\begin{align*}
|c^{\g}_n|:=\frac{1}{n!}\Big|\Big((1+x)^n(1+\g x)^{-n-1}\Big)^{(n)}\Big|_{x=0}\Big|\leq 2
\end{align*}
and $|c^{0}_n|=|c^{1}_n|= 1$.
\end{lemma}

\begin{proof}
Fix some $n\in\mathbb{N}$. We have

\begin{align*}
h(x):=(1+x)^n(1+\g x)^{-n-1}=\left(\frac{1}{\g}+\frac{1-\frac{1}{\g}}{1+\g x}\right)^n\frac{1}{1+\g x}=\sum_{g=0}^n\binom{n}{g}\left(\frac{1}{\g}\right)^g\frac{\left(1-\frac{1}{\g}\right)^{n-g}}{(1+\g x)^{n-g+1}},
\end{align*}
whence

\begin{align*}
\frac{1}{n!}h^{(n)}(x)=\frac{1}{n!}\sum_{g=0}^n \binom{n}{g}\frac{\left(1-\frac{1}{\g}\right)^{n-g}}{\g^g}\frac{(-1)^n\g^n}{(1+\g x)^{2n-g+1}}\frac{(2n-g)!}{(n-g)!}.
\end{align*}

Making the change of variable $r=n-g$, we obtain

\begin{align*}
&\frac{1}{n!}h^{(n)}(0)=\frac{(-1)^n}{n!}\sum_{r=0}^n (\g-1)^{r}\binom{n}{n-r}\frac{(n+r)!}{r!}=(-1)^n\sum_{r=0}^n(1-\g)^r(-1)^r\frac{(n+r)!}{(n-r)!(r!)^2}\\
&=(-1)^n\sum_{r=0}^n(1-\g)^r(-1)^r \binom{n+r}{r}\binom{n}{n-r}=(-1)^n\sum_{r=0}^n(1-\g)^r \binom{-n-1}{r}\binom{n}{n-r}\\
&=(-1)^n\frac{1}{n!}\Big((1+x)^n(1+(1-\g) x)^{-n-1}\Big)^{(n)}\Big|_{x=0}.
\end{align*}

Hence, $c^{\g}_n=(-1)^nc^{1-\g}_n$. Therefore, it is enough to prove the claim only for $\g\in (0,1/2]$, and separately, for $\g=0$. 

Let $\g\in(0,1/2]$. Note that the function $h(x):=(1+x)^n(1+\g x)^{-n-1}$ is analytic inside the circle of radius $1/\sqrt{\g}$. Let $x=(a+bi)/\sqrt{\g},\;a^2+b^2=1$, then

\begin{align*}
\bigg|\frac{1+x}{1+\g x}\bigg|^2=\frac{1+\frac{a^2}{\g}+\frac{2a}{\sqrt{\g}}+\frac{b^2}{\g}}{1+\g a^2+2\sqrt{\g}a+\g b^2}=\frac{1}{\g}.
\end{align*}

Thus, the maximum of the function $h(x)$ on the circle $|x|=1/\sqrt{\g}$ cannot exceed $\g^{-n/2}/(1-\g)$ and due to Cauchy's inequalities,

\begin{align*}
|c_n^{\gamma}|\leq (1/\sqrt{\g})^{-n}\g^{-n/2}/(1-\g)=(1-\g)^{-1}\leq 2.
\end{align*} 
For $\g=0$, we have $h(x)=(1+x)^n=\sum_{k=0}^{n}\binom{n}{k} x^k$, whence $c_n=1$.
\end{proof}

Thus, Lemma \ref{g} gives

\begin{align}\label{S_3}
|S_3|\leq 2.
\end{align} 

{\bf Step 2. Estimating $S_4$.} For any $j=1,...,t$, there holds

\begin{align}\label{appr}
0&<\Big(\frac{k+1}{r}\Big)^j-\frac{k...(k-j)}{(r-1)...(r-j)}<\Big(\frac{k+1}{r}\Big)^{j}-\Big(\frac{k+1-t}{r}\Big)^j<\frac{j(k+1)^{j-1}}{r^{j}}\leq \frac{1}{r}.
\end{align}
By Stirling's formula,

\begin{align}\label{St}
\binom{t}{j}\binom{t+j}{j}=\frac{(t+j)!}{(t-j)!(j!)^2}<\sqrt{t+j}\frac{(t+j)^{t+j}}{(t-j)^{t-j}j^{2j}}&\leq \sqrt{t+\frac{t}{\sqrt{2}}}\left(\frac{\big(1+\frac{1}{\sqrt{2}}\big)^{1+\frac{1}{\sqrt{2}}}}{\big(1-\frac{1}{\sqrt{2}}\big)^{1-\frac{1}{\sqrt{2}}}2^{-\frac{1}{2\sqrt{2}}}}\right)^t\nonumber\\
&=\sqrt{t+\frac{t}{\sqrt{2}}}L^t.
\end{align}

Combining \eqref{appr} and \eqref{St}, we get

\begin{align*}
|S_4|\leq t\cdot \sqrt{t+\frac{t}{\sqrt{2}}}L^t\cdot\frac{1}{r}<2L^{t}\frac{t^{1.5}}{r}<1
\end{align*}
for $r\geq 2L^tt^{1.5}$, which along with \eqref{S_3} gives us relation \eqref{str3}.
\end{proof}

Let us turn back to the entries of the matrix $T^{\dagger}$. In view of \eqref{Ree}, we derive from \eqref{uff} and Lemma \ref{str}

\begin{align*}
&|(Y^*\tilde{N}_pW^{\dagger})_{2u,k}|<2\cdot 4^{-u}\sum_{v=0}^u\binom{2u}{u-v}\sum_{w=0}^{2p-1}\frac{(v+w)!}{v!w!}\sum_{t=w}^{2p-1}\frac{\binom{t+w}{w}\binom{r-w-1}{r-t-1}}{\binom{2t}{t}\binom{r+t}{2t+1}}4\binom{r-1}{t}\tau(k)\\
&=8(r-1)!\sum_{v=0}^u\sum_{w=0}^{2p-1}\sum_{t=w}^{2p-1}\frac{(u!)^2(v+w)!(t+w)!(r-w-1)!(2t+1)}{(u-v)!(u+v)!v!(w!)^2(r-t-1)!(t-w)!(r+t)!}\tau(k),
\end{align*} 
where
\begin{equation*}
\tau(k)=\begin{cases}
1,\quad\text{if}\;k\leq 2p-1\;\text{or}\;k\geq r-2p,\\
r^{\frac{-q+2p-1}{2}},\quad\text{if} \;k= q\;\text{or}\;k=r-1-q,\;2p-1\leq q\leq\frac{\sqrt{r}}{2},\\
r^{\frac{-\sqrt{r}/2+2p-1}{2}},\quad\text{if} \;\frac{\sqrt{r}}{2} <k<r-1-\frac{\sqrt{r}}{2}.
\end{cases}
\end{equation*}

Going from $w-1$ to $w$, the corresponding product changes in 
\begin{align*}
\frac{(v+w)(t+w)(t-w+1)}{(r-w)w^2}<\frac{(4p)^3}{r-2p}<1
\end{align*}
times, hence, the maximum is attained at $w=0$. So,

\begin{align*}
&|(Y^*\tilde{N}_pW^{\dagger})_{2u,k}|<8(r-1)!\;2p\sum_{v=0}^u\sum_{t=0}^{2p-1}\frac{(u!)^2(r-1)!(2t+1)}{(u-v)!(u+v)!(r-t-1)!(r+t)!}\tau(k)\\
&<16p\cdot 4p\cdot(u+1)\sum_{t=0}^{2p-1}\frac{((r-1)!)^2}{(r-t-1)!(r+t)!}\tau(k)<16p\cdot 4p\cdot p\cdot 2p\frac{1}{r}\tau(k)=\frac{128p^4\tau(k)}{r}.
\end{align*}
Similarly, from \eqref{W} we have 

\begin{align*}
|(W^{\dagger})_{qk}|&\leq \sum_{w=q}^{2p-1}(w+1)\sum_{t=w}^{2p-1}\frac{\binom{t+w}{w}\binom{r-w-1}{r-t-1}}{\binom{2t}{t}\binom{r+t}{2t+1}}4\binom{r-1}{t}\tau(k)\\
&=\sum_{w=q}^{2p-1}(w+1)(r-1)!\sum_{t=w}^{2p-1}\frac{(t+w)!(r-w-1)!(2t+1)}{w!(r-t-1)!(t-w)!(r+t)!}\tau(k).
\end{align*}
Here going from $w-1$ to $w$ our term changes in $(t+w)(t-w+1)/w(r-w)<2(2p)^2/(r-2p)<1$ times, so,

\begin{align}\label{W_est}
|(W^{\dagger})_{qk}|\leq 2p\cdot 2p\cdot 4p\sum_{t=0}^{2p-1}\frac{((r-1)!)^2}{(r-t-1)!(r+t)!}\tau(k)<\frac{32p^4\tau(k)}{r}.
\end{align}

Since an entry of $Y^*$ does not exceed in absolute value $2p\cdot p^{4p-2}$ (see \eqref{yy}), then an entry of $Y^*\tilde{N}_pX$ is less than or equal in absolute value to $2p\cdot 2p^{4p-1}\cdot 8p^2s^{4p-1}/(r-8p^2s^{4p-1})$ (here we used estimate \eqref{X}). So an entry of $Y^*\tilde{N}_pXW^{\dagger}$ does not exceed $2p\cdot 32p^{4p+2}(8p^2s^{4p-1}/(r-8p^2s^{4p-1}))\cdot 32p^4/r$. Thus, for $r-1\geq k>s-1$, according to \eqref{WW} we have

\begin{align}\label{YV_final_1}
|(Y^*\tilde{N}_pV^{\dagger})_{2u,k-s}|=|(Y^*\tilde{N}_p(E+X)W^{\dagger})_{2u,k}|&\leq |(Y^*\tilde{N}_pW^{\dagger})_{2u,k}|+|(Y^*\tilde{N}_pXW^{\dagger})_{uk}|\nonumber\\
&<\frac{2^7p^4\tau(k)}{r}+\frac{2^{15}p^{4p+9}s^{4p-1}}{(r-8p^2s^{4p-1})r}.
\end{align}
For $r-1\geq k>s-1$ and an odd $u$, due to \eqref{W_est}
\begin{align}\label{YV_final_2}
|(Y^*\tilde{N}_pV^{\dagger})_{u,k-s}|\leq \max\limits_{i}|(V^{\dagger})_{i,k-s}|\leq 2\max\limits_{i}|(W^{\dagger})_{ik}|<2 \frac{32p^4\tau(k)}{r}.
\end{align}
From \eqref{YV_final_1} and \eqref{YV_final_2} we finally get
\begin{align*}
\|Y^*\tilde{N}_pV^{\dagger}\|_{\infty}&\leq r\cdot\frac{2^{15}p^{4p+9}s^{4p-1}}{(r-8p^2s^{4p-1})r}+\frac{2^7p^4}{r}\Big(\sum_{k=0}^{2p-1}+\sum_{k=r-2p}^{r-1}+\sum_{k=2p}^{\lfloor \sqrt{r}/2\rfloor}+\sum_{k=r-1-\lfloor \sqrt{r}/2\rfloor}^{r-2p-1}+\sum_{k=\lfloor \sqrt{r}/2\rfloor+1}^{r-\lfloor \sqrt{r}/2\rfloor-2}\Big)\tau(k)\\
&<\frac{2^{15}p^{4p+9}s^{4p-1}}{r-8p^2s^{4p-1}}+\frac{2^7p^4}{r}(2p+2p+1+1+r\cdot r^{-\frac{\sqrt{r}}{4}+p-\frac{1}{2}})<\frac{2^{16}p^{4p+9}s^{4p-1}}{r},
\end{align*}
whence in light of \eqref{dag} condition \eqref{cond_2} follows, and the needed is proved.
\end{proof}

\begin{acknowledgements} I am grateful to Lavrentin Arutyunyan for the lively and fruitful discussion which resulted in the proof of Lemma \ref{g}.
\end{acknowledgements}


\begin{thebibliography}{7}

\bibitem[1]{BG} Ben-Israel A., Greville T.N.E., {\it Generalized inverses: Theory and applications}, 2nd edition, Springer, New York, 2003.

\bibitem[2]{BP} Boyd J.P., Petschek R., {\it The relationships between Chebyshev, Legendre and Jacobi polynomials: The generic superiority of Chebyshev polynomials and three important exceptions}, J. Sci. Comput. 59  (2014), 1--27.

\bibitem[3]{EPS} Eisinberg A., Pugliese P., Salerno N., {\it Vandermonde matrices on integer nodes: the rectangular case}, Numer. Math. 87  (2001), 663--674.

\bibitem[4]{I} Ismail M.E.H., {\it Classical and quantum orthogonal polynomials in one variable}, Encyclopedia of Mathematics and its Applications, vol. 98, Cambridge: Cambridge University Press, 2005.

\bibitem[5]{MS}
Macon N., Spitzbart A., {\it Inverses of Vandermonde matrices}, The American Mathematical Monthly, 65(2)  (1958), 95--100.

\bibitem[6]{HM}
Mason J.C., Handscomb D.C., {\it Chebyshev polynomials}, Chapman and Hall/CRC (2002).

\bibitem[7]{PK} 
Pantelous A.A., Karageorgos A.D., {\it Generalized inverses of the vandermonde matrix: Applications in control theory}, Int. J. Control Autom. Syst. 11  (2013), 1063--1070.


\bibitem[8]{T} 
Tetunashvili T., {\it Universal series and subsequences of functions}, Sb. Math., 209(10)  (2018), 1498--1532.

\bibitem[9]{V} 
Verde-Star L., {\it Inverses of generalized Vandermonde matrices}, J. Math. Anal. Appl., 131(2)  (1988), 341--353.
\end{thebibliography}
\end{document}